\documentclass[11pt,reqno,tbtags,a4paper]{amsart}
\usepackage{amssymb}
\usepackage{url}
\usepackage[square,numbers]{natbib}
\bibpunct[, ]{[}{]}{;}{n}{,}{,}

\title[Preferential attachment trees are  split trees]
{Random recursive trees and preferential attachment trees are random split
  trees} 

\date{16 June, 2017}

\author{Svante Janson}
\thanks{Partly supported by the Knut and Alice Wallenberg Foundation}
\address{Department of Mathematics, Uppsala University, PO Box 480,
SE-751~06 Uppsala, Sweden}
\email{svante.janson@math.uu.se}
\urladdr{http://www.math.uu.se/svante-janson}

\subjclass[2010]{} 

\numberwithin{equation}{section}

\renewcommand\le{\leqslant}
\renewcommand\ge{\geqslant}

\allowdisplaybreaks

\theoremstyle{plain}
\newtheorem{theorem}{Theorem}[section]
\newtheorem{lemma}[theorem]{Lemma}

\newtheorem{corollary}[theorem]{Corollary}

\theoremstyle{definition}
\newtheorem{example}[theorem]{Example}

\newtheorem{remark}[theorem]{Remark}

\theoremstyle{remark}

\newenvironment{romenumerate}[1][-10pt]{
\addtolength{\leftmargini}{#1}\begin{enumerate}
 }{\end{enumerate}}

\newenvironment{romxenumerate}[1][-10pt]{
\addtolength{\leftmargini}{#1}\begin{enumerate}
 }{\end{enumerate}}

\newcommand{\refT}[1]{Theorem~\ref{#1}}

\newcommand{\refL}[1]{Lemma~\ref{#1}}
\newcommand{\refR}[1]{Remark~\ref{#1}}
\newcommand{\refS}[1]{Section~\ref{#1}}

\newcommand{\refE}[1]{Example~\ref{#1}}

\newcommand{\refApp}[1]{Appendix~\ref{#1}}

\newcommand\REM[1]{{\raggedright\texttt{[#1]}\par\marginal{XXX}}}
\newcommand\XREM[1]{\relax}

\begingroup
  \divide\count255 by 60
  \multiply\count255 by -60
  \advance\count255 by \time
  \ifnum \count255 < 10 \xdef\klockan{\the\count1.0\the\count255}
  \else\xdef\klockan{\the\count1.\the\count255}\fi
\endgroup

\newcommand{\sumi}{\sum_{i=1}^\infty}

\newcommand{\summ}{\sum_{m=1}^\infty}

\newcommand{\sumim}{\sum_{i=1}^m}

\newcommand{\sumkn}{\sum_{k=1}^n}

\newcommand\set[1]{\ensuremath{\{#1\}}}

\newcommand\bigpar[1]{\bigl(#1\bigr)}
\newcommand\Bigpar[1]{\Bigl(#1\Bigr)}

\newcommand\xcpar[1]{\{#1\}}

\def\rompar(#1){\textup(#1\textup)}    

\def\xexp(#1){e^{#1}}
\newcommand\ceil[1]{\lceil#1\rceil}

\newcommand\ntoo{\ensuremath{{n\to\infty}}}

\newcommand\punkt{.\spacefactor=1000}    
\newcommand\iid{i.i.d\punkt}    
\newcommand\ie{i.e\punkt}
\newcommand\eg{e.g\punkt}

\newcommand\cf{cf\punkt}
\newcommand{\as}{a.s\punkt}

\newcommand{\tend}{\longrightarrow}

\newcommand\asto{\overset{\mathrm{a.s.}}{\tend}}
\newcommand\eqd{\overset{\mathrm{d}}{=}}

\newcommand\bbN{\mathbb N}

\newcounter{CC}
\newcounter{cc}

\newcommand\E{\operatorname{\mathbb E{}}}
\renewcommand\P{\operatorname{\mathbb P{}}}

\newcommand\ga{\alpha}
\newcommand\gb{\beta}
\newcommand\gd{\delta}

\newcommand\gam{\gamma}
\newcommand\gG{\Gamma}

\newcommand\gth{\theta}

\renewcommand\phi{\xxx}  

\newcommand\cP{\mathcal P}

\newcommand\cT{{\mathcal T}}

\newcommand\cV{\mathcal V}

\newcommand\ett[1]{\boldsymbol1\xcpar{#1}}

\newcommand\qww{^{-2}}

\newcommand\oi{\ensuremath{[0,1]}}

\newcommand\dd{\,\mathrm{d}}

\newcommand\vv{^{(v)}}
\newcommand\Tb{\cT_b}
\newcommand\GEM{\operatorname{GEM}}
\newcommand\PD{\operatorname{PD}}
\newcommand\Dir{\operatorname{Dir}}
\newcommand\XP{^{\chi,\rho}}
\newcommand\XM[1]{^{-1,#1}}
\newcommand\VP{^{\cP}}
\newcommand\pa{preferential attachment}
\newcommand\lpa{linear \pa}
\newcommand\bP{\mathbf{P}}
\newcommand\aP{P}
\newcommand\caP{\cP}
\newcommand\hP{\hat P}
\newcommand\bX{\mathbf X}
\newcommand\ao{_1^\infty}
\newcommand\mary{$m$-ary}
\newcommand\maryt{$m$-ary  tree}
\newcommand\maryit{$m$-ary increasing tree}
\newcommand\hY{\widehat Y}
\newcommand\xH{H^*}
\newcommand\hW{\hat W}

\newcommand{\Polya}{P\'olya}

\begin{document}

\begin{abstract} 
We consider linear  preferential attachment trees, and show that they
can be regarded as random split trees in the sense of Devroye (1999),
although with infinite potential  branching.
In particular, this applies to
the random recursive tree and the standard preferential attachment tree.
An application is given to the sum over all pairs of nodes of the common
number of ancestors.  
\end{abstract}

\maketitle

\section{Introduction}\label{S:intro}

The purpose of this paper is to show that the linear  preferential
attachment trees, a class of random trees that includes and generalises
both the random
recursive tree and the standard preferential attachment tree,
can be regarded as random split trees in the sense of \citet{Devroye},
although with infinite (potential)  branching.

Recall that the random recursive tree
is an unordered rooted tree that is
constructed by adding nodes one by one, with each node attached as the child
of an existing node chosen uniformly at random; see \eg{} \cite[Section
1.3.1]{Drmota}. 
The general preferential attachment tree is constructed in a
similar way, but for each new node, its parent is chosen among the existing
nodes with the probability of choosing a node $v$ proportional to
$w_{d(v)}$, where $d(v)$ is the outdegree (number of existing children) of
$v$, and $w_0,w_1,\dots$ is a given sequence of weights.
The constant choice $w_k=1$ thus gives the random recursive tree. 
The preferential attachment tree 
made popular by \citet{BarabasiA} (as a special case of  more general
preferential attachment graphs)
is given by the choice $w_k=k+1$; 
this coincides with the 
plane oriented recursive tree earlier introduced by \citet{Szymanski}.
We shall here consider the more general linear case
\begin{equation}\label{chirho}
  w_k=\chi k+\rho
\end{equation}
for some real parameters $\chi$ and $\rho>0$,
which was introduced (at least for $\chi\ge0$) by \citet{Pittel}.
Thus the random recursive tree is obtained for $\chi=0$ and $\rho=1$, while
the standard preferential attachment tree is the case $\chi=\rho=1$.
We allow $\chi<0$, but in that case we have to assume that $\rho/|\chi|$ is
an integer, say $m$, in order to avoid negative weights. (We then have
$w_m=0$ so a node never gets more than $m$ children, and $w_k$ for $k>m$ are
irrelevant; see further \refS{S<0}.) 
See also \cite[Section 6]{SJ306} and the further references given there.
We denote the random linear preferential attachment tree with $n$ nodes and
weights \eqref{chirho} by $T\XP_n$.

\begin{remark}\label{R1}
Note that multiplying all
$w_k$ by the same positive constant will not change the trees, so only the
ratio $\chi/\rho$ is important. Hence we may normalize the parameters in
some way when convenient; however, different normalizations are convenient
in different situations, and therefore we keep the general and more flexible
assumptions above unless we say otherwise.

Note also that our assumptions imply $w_1=\chi+\rho>0$ except in the  case
$\chi=-\rho$, when $w_1=0$ and $T_n\XP$ deterministically is a path.
We usually ignore that trivial case in the sequel, and assume $\chi+\rho>0$.
\end{remark}

\begin{remark}
  The three cases $\chi>0$, $\chi=0$ and $\chi<0$ 
give the three classes of very simple increasing trees
defined and characterized by \citet{PP-verysimple}, see also
\cite{BergeronFS92} and \cite[Section 1.3.3]{Drmota}.
In fact, it suffices to consider $\chi=1$, $\chi=0$ and $\chi=-1$, see
\refR{R1}.
Then, $\chi=0$ yields the random recursive tree, as said above;
$\chi=1$ yields the generalised plane oriented recursive tree; $\chi=-1$
(and $\rho=m\in\bbN$) yields the \maryit{}, see further
\refS{S<0}. 
\end{remark}

Random split trees were defined by \citet{Devroye} as rooted trees
generated by  
a certain recursive procedure using a stream of balls added to the root.
We only need a simple but important special case (the case $s=1$, $s_0=1$,
$s_1=0$ in the notation of \cite{Devroye}), in which case the general
definition simplifies to the following (we use $\cP$ and $P_i$ instead of
$\cV$ and $V_i$ in \cite{Devroye}):

Let $b\ge2$ be fixed and let $\cP=(P_i)_1^b$ be a random vector of
probabilities: in other words, $P_i\ge0$ and $\sum_{i=1}^b P_i=1$.
Let  
$\Tb$ be the infinite rooted tree where each node has $b$ children,
labelled $1,\dots,b$, and give each node $v\in \Tb$ an independent copy
$\cP\vv=(P_i\vv)_1^b$ of $\cP$. 
(These vectors are thus random, but chosen only once and
fixed during the construction.)
Each node in $\Tb$ may hold one ball; if it does, we say that the node is
\emph{full}.  Initially all nodes are empty.
Balls arrive, one by one, to the root  of $\Tb$, and move
(instantaneously) according to the following rules.
\begin{romenumerate}
\item \label{split1}
A ball arriving at an empty node stays there, making the node full.
\item \label{split2}
A ball arriving at a node $v$ that already is full
continues to a child of $v$; the child is chosen at random, with child $i$
chosen with probability $P_i\vv$. Given the vectors $\cP\vv$, all these
choices are made independently of each other.
\end{romenumerate}
The random split tree $T_n=T_n\VP$ is the subtree of $\Tb$ consisting of the
nodes that contain the first $n$ balls.
Note that the parameters apart from $n$ 
in (this version of) the construction are
$b$ and
the random $b$-dimensional probability vector $\cP$ (or rather its
distribution); $\cP$ is called the \emph{split vector}.

\citet{Devroye} gives several examples of this construction (and also of other
instances of his general definition). 
One of them is the random binary search tree,  
which is obtained with $b=2$ and $\cP=(U,1-U)$, with $U\sim U(0,1)$, the
uniform distribution on $\oi$.
The main purpose of the definition of random split trees is that they
encompass many 
different examples of random trees that have been studied separately; 
the introduction of split trees made it possible to treat them together.
Some general results were proved in \cite{Devroye}, and further
results and examples have been added by other authors, see for example
\cite{BroutinH,Holmgren}.

\citet{Devroye} considers only finite $b$, yielding trees $T_n$ where each
node has at most $b$ children, but the definition above of random split
trees  extends to 
$b=\infty$, when each node can have an unlimited number of children.
This is the case that we shall use. (Note that random recursive trees and
linear preferential attachment trees with $\chi>0$
do not have bounded degrees; 
see \refS{S<0} for the case $\chi<0$.)
Our purpose is to show that with this extension, 
also
linear preferential attachment trees
are random split trees.

\begin{remark}\label{Rlabel}
The general preferential attachment tree is usually considered as an
unordered tree. 
However, it is often convenient to label the children of each node by
$1,2,3,\dots$ in the order that they appear; hence we can also regard the
tree as a ordered tree. Thus both the preferential attachment trees and the
split trees considered in the present paper can be regarded as subtrees of
the infinite  Ulam--Harris(--Neveu) tree $\cT_\infty$, which is the infinite
rooted 
ordered tree where every node has a countably infinite set of children,
labelled $1,2,3,\dots$. (The nodes of $\cT_\infty$ are all finite strings
$\iota_1\dots\iota_m \in \bbN^*:=\bigcup_0^\infty\bbN^m$, with the empty string as
the root.)

One advantage of this is that it makes it possible to talk unambiguously
about inclusions among the trees.
We note that both constructions above yield random sequences of
trees $(T_n\XP)_{n=1}^\infty$ and $(T_n\VP)_{n=1}^\infty$ that are increasing:
$T\XP_n\subset T\XP_{n+1}$ and $T\VP_n\subset T\VP_{n+1}$. 
\end{remark}

\begin{remark}
  \label{Rsplit}
The random split tree, on the other hand, is defined as an ordered tree,
with the potential children of a node labelled $1,2,\dots$. Note that these
do not 
have to appear in order; child 2 may appear before child 1, for example.

We can always consider the random split tree as unordered by ignoring the
labels. If we do so, any (possibly random) permutation of the random
probabilities $P_i$ yields the same unordered split tree.
(In particular, if $b$ is finite, then it is natural to permute $(P_i)_1^b$
uniformly at random, thus making all $P_i$ having the same (marginal)
distribution \cite{Devroye}. However, we cannot do that when $b=\infty$.) 
\end{remark}

Using the GEM and Poisson--Dirichlet distributions defined in \refS{Snot},
we can state our main result as 
follows. The proof is given in \refS{Spf}, using Kingman's paintbox
representation of exchangeable partitions. (\refApp{AT} gives an
alternative, but related, argument using exchangeable sequences instead.)
In fact, the result can be said to be implicit in \cite{Pitman} and
\cite{Bertoin}, see \eg{} \cite[Corollary 2.6]{Bertoin}.

\begin{theorem}\label{T1}
Let  $(\chi,\rho)$ be as above, and assume $\chi+\rho>0$.
Then,
provided the trees are regarded as unordered trees,
the linear preferential attachment tree $T_n\XP$ has, for every $n$,
the same distribution as the random split tree $T_n\VP$
with $b=\infty$ and $\cP\sim\GEM\bigpar{\chi/(\chi+\rho),\rho/(\chi+\rho)}$,

Moreover, 
(re)la\-belling the children of each node in order of appearance,
the sequences $(T_n\XP)_1^\infty$  and $(T_n\VP)_1^\infty$ of random trees
have the same distribution.

The same results hold also if we instead let $\cP$ have the Poisson--Dirichlet
distribution $\PD\bigpar{\chi/(\chi+\rho),\rho/(\chi+\rho)}$.
\end{theorem}

The result extends  to the trivial case $\chi+\rho=0$, with
$\cP\sim\GEM(0,0)=\PD(0,0)$, \ie, $P_1=1$; in this case $T_n$ is a path.

\begin{corollary}\label{CR}
The sequence of random recursive trees $(T_n)\ao=(T_n^{0,1})\ao$ has
the same distribution as the sequence of random split trees $(T_n\VP)\ao$
with  $\cP\sim\GEM(0,1)$
or $\cP\sim\PD(0,1)$ (as unordered trees).
\end{corollary}

Recall that the split vector $\PD(0,1)$ appearing here also appears as, for
example, the asymptotic distribution of the (scaled) sizes of the cycles in
a random permutation; see \eg{} \cite[Section 3.1]{Pitman}.

\begin{corollary}\label{CPA}
The sequence of standard \pa{}  trees $(T_n)\ao\allowbreak=(T_n^{1,1})\ao$ has
the same distribution as the sequence of random split trees $(T_n\VP)\ao$
with  $\cP\sim\GEM(\frac12,\frac12)$
or $\cP\sim\PD(\frac12,\frac12)$ (as unordered trees). 
\end{corollary}

Note that in \refT{T1} and its corollaries above, it is important that
we ignore 
the original labels, and either regard the trees as unordered, or
(re)label the children of each node in order of appearance (see
\refR{Rlabel}); random split trees with the original labelling are
different (see \refR{Rsplit}).
In the case $\chi<0$, there is also a version for labelled trees, see 
\refT{T2}.

We give an application of \refT{T1} in \refS{Sapp}.

\section{Notation}\label{Snot}

If $T$ is a rooted tree, and $v$ is a node in $T$, then $T^v$ denotes the
subtree of $T$ consisting of $v$ and all its descendants. (Thus $T^v$ is
rooted at $v$.)

A \emph{principal subtree} (also called branch)
of $T$ is a subtree $T^v$ where $v$ is a child of
the root $o$ of $T$. Thus the node set $V(T)$ of $T$
is partitioned into $\set{o}$
and the node sets $V(T^{v_i})$ of the principal subtrees.

For a (general) preferential attachment tree, with a given weight
sequence $(w_k)_k$,
the weight of a node $v$ is $w_{d(v)}$, where $d(v)$ is the outdegree of
$v$. The (total) weight $w(S)$ of a set $S$
of nodes is the sum of the weights of the nodes in $S$; if $T'$ is a tree,
we write $w(T')$ for $w(V(T'))$.

The Beta distribution $B(\ga,\gb)$ 
is for  $\ga,\gb>0$, as usual, the distribution on $\oi$
with density function $c x^{\ga-1}(1-x)^{\gb-1}$, with the normalization
factor $c=\gG(\ga+\gb)/\bigpar{\gG(\ga)\gG(\gb)}$. 
We allow also the limiting cases
$B(0,\gb):=\gd_0$ ($\gb>0$)  and
$B(\ga,0):=\gd_1$ ($\ga>0$), \ie, the distributions of the deterministic
variables 0 and 1, respectively.

The GEM distribution $\GEM(\ga,\theta)$ is the distribution of a random
infinite vector of probabilities $(P_i)_1^\infty$ that can be represented as
\begin{equation}\label{gem}
  P_i=Z_i\prod_{j=1}^{i-1}(1-Z_j),
\qquad j\ge1,
\end{equation}
where the $Z_j$ are independent random variables with Beta distributions
\begin{equation}
  \label{gemZ}
Z_j\sim B(1-\ga,\theta+j\ga). 
\end{equation}
Note that \eqref{gem} has the interpretation that $P_1=Z_1$, $P_2$ is a
fraction $Z_2$ of the remaining probability $1-P_1$, $P_3$ is a fraction
$Z_3$ of the remainder $1-P_1-P_2=(1-Z_1)(1-Z_2)$, and so on.
Here the parameters $\ga$ and $\theta$ are assumed to satisfy
$-\infty<\ga<1$ and $\theta+\ga\ge0$; furthermore, if $\ga<0$, then
$\gth/|\ga|$ has to be an integer.
(If $\ga<0$ and $\gth=m|\ga|$, then $Z_m=1$, and thus \eqref{gem} yields
$P_{i}=0$ for all $i>m$; hence it does not matter that $Z_j$ really is
defined only for $j\le m$ in this case.)
See further \eg{} \cite[Section 3.2]{Pitman}.

The Poisson--Dirichlet distribution $\PD(\ga,\theta)$ is the distribution of
the random infinite vector $(\hat P_i)\ao$ obtained by reordering
$(P_i)\ao\sim\GEM(\ga,\theta)$  in decreasing order.

\section{Proof of \refT{T1}}\label{Spf}

\begin{lemma}\label{L1}
With the linear weights \eqref{chirho}, a tree $T$ with $m$ nodes has total
weight $w(T)=(m-1)\chi+m\rho=m(\chi+\rho)-\chi$.
\end{lemma}

\begin{proof}
  Let the nodes have outdegrees $d_1,\dots,d_m$.
Then $\sumim d_i=m-1$, and the weight of the tree is thus
\begin{equation*}
w(T)=
  \sumim (\chi d_i+\rho)
=  \chi \sumim  d_i+m\rho=(m-1)\chi+m\rho.
\qedhere
\end{equation*}
\end{proof}
 
\begin{lemma}\label{L2}
  Consider the sequence of \lpa{} trees $(T_n)_1^\infty=(T\XP_n)_1^\infty$,
  with the 
  children of the root labelled in order of appearence.
Let $N_j(n):=|T^j_n|$,  the size of the $j$-th principal subtree of $T_n$.
Then $N_j(n)/n\to P_j$ \as{} as \ntoo, for every $j\ge1$ and some random
variables $P_j$
with the distribution 
$\GEM\bigpar{\chi/(\chi+\rho),\rho/(\chi+\rho)}$.
(In the trivial case $\chi+\rho=0$, interpret this as $\GEM(0,0)$.)
\end{lemma}

\begin{proof}
The case $\chi+\rho=0$ is trivial, with $N_1(n)=n-1$ and $P_1=1$.
Hence we may assume that $\chi+\rho>0$.
Furthermore, see \refR{R1}, we 
may and shall, for convenience, assume that 
\begin{equation}\label{=1}
\chi+\rho=1.  
\end{equation}

The lemma now follows from \citet[Theorem 3.2]{Pitman}, which is stated for 
``the Chinese restaurant with the $(\ga,\theta)$ seating plan'',
since we may regard the principal subtrees as tables in a Chinese restaurant
(ignoring the root), and then the \pa{} model with \eqref{chirho} translates
into the $(\chi,\rho)$ seating plan as defined in \cite{Pitman}.
(Cf.\ the bijection between recursive trees and permutations in
\cite[Section 6.1.1]{Drmota}, which yields this correspondence;
the uniform case treated there is the case $(\chi,\rho)=(0,1)$, which yields
the usual Chinese restaurant process.)

For completeness, we give a direct proof using \Polya{} urns in \refApp{AL2}.
\end{proof}

\begin{proof}[Proof of \refT{T1}]
  Recall that in the (general) preferential attachment tree, the parent $u$ of a
new node
is chosen to be a  node $v$
with probability proportional to the current weight $w_{d(v)}$
of the node. We can make this random choice in several steps, by first
deciding randomly
whether $u$ is the root or not, and if not, which principal
subtree it belongs to, making this choice with probabilities proportional to
the total weights of these sets of nodes. If $u$ is chosen to be in a
subtree $T^w$, we then continue recursively inside this tree, by deciding
randomly whether $u$ is the root of $T^w$ or not, and if not, which principal
subtree of $T^w$ it belongs to, again with probabilities proportional to
the total weights, and so on.

Consequently, the general preferential attachment tree can be constructed 
recursively using a stream of new nodes (or balls) similarly to the random
split tree, with the rules:
\begin{romxenumerate}
\item \label{pax1}
A ball arriving at an empty node stays there, making the node full.
\item \label{pax2}
A ball arriving at a node $v$ that already is full
continues to a child of $v$. The child is chosen at random;
if $v$ has $d$ children $v_1,\dots,v_d$, then
the ball is passed to child $i$ with probability $cw(T^{v_i})$ for each
$i=1,\dots,m$, and to the new child $m+1$ with probability $cw(v)=c(\chi
d+\rho)$, where $c=1/w(T^v)$ is a positive normalization factor.
\end{romxenumerate}

Thus both the random split trees and the linear preferential attachment trees
can be constructed recursively, and in order to show \refT{T1},
it suffices to show that the two constructions yield the same result at the
root, i.e., that balls after the first are passed on to the
children of the root in the same way in both random trees.
(Provided we ignore the order of the children, or (re)label the children in
order of appearance.)

Consider the linear preferential attachment tree with the construction
above. As in the proof of \refL{L2}, we may assume that \eqref{=1} holds.

Label the children of the root in order of appearance, see \refR{Rlabel}.
The first ball stays at the root, while all others are passed on; we
label each ball after the first by the label of the child of the root that
it is passed to.
This gives a random sequence $(X_i)_{i=1}^\infty$ of labels in $\bbN$,
(where $X_i$ is the label of ball $i+1$, the $i$th ball that is passed on).
By construction, the random sequence $(X_i)_i$ is such that the
first 1 appears before the first 2, which comes before the first 3, and so on;
we call a finite or infinite sequence $(x_i)_i$ of labels in $\bbN$
\emph{acceptable} if it has this property. 

Let
$(x_i)_1^n$ be a finite acceptable sequence of length $n\ge0$, and let $n_k$ be the
number of times $k$ appears in the sequence; further, let $d_n$ be the largest
label in the sequence, so $n_k\ge1$ if $1\le k\le d_n$, but $n_k=0$ if $k>d$.
If $(X_i)_1^n=(x_i)_1^n$, then the subtree $T^k$ with label $k$ has $n_k$ nodes,
and thus by \refL{L1} and our assumption \eqref{=1} 
weight $n_k(\chi+\rho)-\chi=n_k-\chi$, provided $k\le d_n$, while
the root has weight $\chi d_n+\rho$. 
Hence, by the construction above, noting that the tree has $n+1$ nodes and
thus by \refL{L1} weight $(n+1)-\chi=n+\rho$,
\begin{equation}
  \P\bigpar{X_{n+1}=k\mid (X_i)_1^n=(x_i)_1^n}=
  \begin{cases}
    (n_k-\chi)/(n+\rho), & 1\le k\le d_n,
\\
(d_n\chi+\rho)/(n+\rho), &  k= d_n+1.
  \end{cases}
\end{equation}
It follows by multiplying these probabilities for $n=0$ to $N-1$
and rearranging factors in the numerator (or by induction) that,
letting $d:=d_N$ and $N_k:=n_k$ for $n=N$, 
\begin{equation}
  \begin{split}
\P\bigpar{(X_i)_1^N=(x_i)_1^N}
= \frac{\prod_{j=0}^{d-1}(j\chi+\rho)\prod_{k=1}^d  \prod_{n_k=1}^{N_k-1}(n_k-\chi)}
{\prod_{n=0}^{N-1}(n+\rho)}.
  \end{split}
\end{equation}
In particular, note that this probability depends on the sequence
$(x_i)_1^N$ only through the numbers $N_k$.
Consequently, if $(x_i')_1^N$ is another acceptable sequence that is a
permutation 
of $(x_i)_1^N$, then 
\begin{equation}
  \label{exch}
\P\bigpar{(X_i)_1^N=(x_i)_1^N}=\P\bigpar{(X_i)_1^N=(x_i')_1^N}.
\end{equation}

Return to the infinite sequence $(X_i)_1^\infty$. 
This sequence encodes a
partition of $\bbN$ into the sets $A_j:=\set{k\in\bbN:X_k=j}$, 
and interpreted in this way, \eqref{exch} says that the random partition
$\set{A_j}_j$ of $\bbN$ is an exchangeable random partition; see \eg{} 
\cite[Section 2.3.2]{Bertoin} or \cite[Chapter 2]{Pitman}.
(See \refApp{AT} for a version of the argument without using the
theory of exchangeable partitions.)
By Kingman's paintbox representation theorem
\cite{Kingman-partition,Kingman-coalescent,Pitman,Bertoin},
any exchangeable random partition of $\bbN$ can be constructed as follows
from some random subprobability vector $(P_i)_1^\infty$, \ie, a
random vector with $P_i\ge0$ and $\sum_iP_i\le1$: Let
$P_\infty:=1-\sum_{i<\infty}P_i\ge0$. Let $Y_i\in\bbN\cup\set\infty$ be
\iid{} random variables with the distribution $(P_i)_1^\infty$.
Then the equivalemce classes are $\set{i:Y_i=k}$ for each $k<\infty$, and the
  singletons $\set{i}$ for each $i$ with $Y_i=\infty$.

In the present case, 
\refL{L2} shows that every principal subtree $T^j$ satisfies either
$|T^j(n)|\to\infty$ 
as \ntoo, or $T^j(n)$ is empty for all $n$ (when $\chi<0$ and $\rho=m|\chi|$
with $m<j$). Hence, the equivalence classes defined by $(X_i)_1^\infty$ are
either empty or infinite, so there are no singletons. Thus $P_\infty=0$, and
$(P_i)_1^\infty$ is a random probability vector.
Moreover, the paintbox construction is precisely what the split tree
construction \ref{split1}--\ref{split2} does at the root, provided we
ignore the labels on the children.

Consequently, the sequence of random split trees $T_n\VP$ with this random
split vector $\cP=(P_i)_1^\infty$  has
the same distribution as the sequence $(T_n\XP)\ao$, provided that
we ignore the labels of the children, or (equivalently) 
relabel the children of a
node in the split trees by their order of appearance.
It remains to identify the split vector $\cP$.

Let $T^j_n$ be the principal subtree of the split tree $T_n\VP$ whose root
is labelled $j$, and let $N_j(n):=|T^j_n|$. Then, by the law of large
numbers, as \ntoo,   
\begin{equation}\label{LLN}
N_j(n)/n\asto P_j,
\qquad j\ge1. 
\end{equation}
Recall that we may permute the probabilities $P_i$ arbitrarily, see
\refR{Rsplit}. Let us relabel the children of the root in their order of
appearance, and permute the $P_i$ correspondingly; thus \eqref{LLN} still holds.
Moreover, we have shown that the tree also can be regarded as a \lpa{} tree,
and with this labelling of the children, \refL{L2} applies.
Consequently, \eqref{LLN} and \refL{L2} yield $(P_i)\ao\sim\GEM(\chi,\rho)$.

Finally, $\PD(\chi,\rho)$ is by definition a permutation of
$\GEM(\chi,\rho)$, and thus these two split vectors define random split
trees with the same distribution (as unordered trees).
\end{proof}

\section{An auxiliary result}\label{Saux}

In the theory of random split trees, an important role is
played by the random variable $W$ defined as a size-biased sample from the
split vector $\cP$; in other words, we first sample $\cP=(P_i)\ao$, then
sample $I\in\bbN$ with the distribution $\P(I=i)=P_i$, and finally let
$W:=P_I$.
Consequently, for any $r\ge0$,
\begin{equation}\label{sizebiased}
\E W^r =\E \sum_i P_i P_i^t=\sum_i \E P_i^{t+1}.  
\end{equation}

We have a simple result for the distribution of $W$ in our case.

\begin{lemma}\label{LW}
  For the random split tree in \refT{T1}, 
$W\sim B\bigpar{\rho/(\chi+\rho),1}$. Thus $W$ has density function 
$\gam x^{\gam-1}$ on $(0,1)$, where $\gam=\rho/(\chi+\rho)$.
\end{lemma}

\begin{proof}
  Let $X_n$ be the number of nodes in $T_n$ that are descendants of the
  first node added after the root.
In the split tree $T_n\VP$, let $I$ be the label of the subtree containing
the first node added after the root.
Conditioned on the split vector $\cP$ at the root, by definition 
$\P(I=i\mid\cP)=P_i$. Furthermore, still conditioned on $\cP$,
the law of large numbers yields that if $I=i$, then $X_n/n\asto P_i$.
Hence, $X_n/n\asto P_I=W$.

On the other hand, in the \pa{} tree $T_n\XP$
with children  labelled in order of appearance, 
the first node after the root always gets label 1 
and thus in the notation of
\refL{L2}, $X_n=N_1(n)$. Consequently, \refL{L2} implies
$X_n/n\asto P_1$. Since \refT{T1} implies that $X_n$ has the same
distribution in the two cases, $W\eqd P_1$.
Furthermore, by \eqref{gem}--\eqref{gemZ}, assuming again for simplicity
\eqref{=1},  
$P_1=Z_1\sim B(1-\chi,\chi+\rho)=B(\rho,1)$. 
\end{proof}

Thus $W\eqd P_1$ for our GEM distribution. This is only a special case of
the general result that rearranging the $P_i$ in size-biased order preserves
$\GEM(\ga,\theta)$ for any pair of parameters, see \cite[Section 3.2]{Pitman}.

\begin{example}
  By \refL{LW} we have $\E W=\gam/(\gam+1)$, and thus by \eqref{sizebiased}
  \begin{equation}\label{hex}
     \sumi \E P_i^2=\E W = \frac{\rho}{\chi+2\rho}.
  \end{equation}
It is possible to calculate the sum in \eqref{hex} directly, using
the definitions \eqref{gem}--\eqref{gemZ}, but the calculation is rather
complicated: 
\begin{equation}
  \begin{split}
\sumi\E P_i^2& = \sumi \E Z_i^2\prod_{j<i}\E(1-Z_j)^2
\\&
=\sumi \frac{(1-\ga)(2-\ga)\prod_{1}^{i-1}(\theta+j\ga)(\theta+1+j\ga)}   
{\prod_{1}^{i}(\theta+1+(j-1)\ga)(\theta+2+(j-1)\ga)} 
\\&
= \frac{(1-\ga)(2-\ga)}{(\theta+1)(\theta+2)}
\sumi \prod_{1}^{i-1}\frac{\theta+j\ga}{\theta+2+j\ga}.
  \end{split}
\end{equation}
The last sum can be summed, for example by writing it as a hypergeometric
function $F(\theta/\ga+1,1;(\theta+2)/\ga+1;1)$ and using Gauss's formula 
\cite[(15.4.20)]{NIST}, leading to \eqref{hex}. 
The proof above seems simpler.
\end{example}

\section{An application}\label{Sapp}

\citet{Devroye} showed general results on the height and insertion depth for
split trees, and used them to give results for various examples.
The theorems in \cite{Devroye}  assume that the split vectors are finite, so
the trees have bounded degrees, but they may be extended to the present
case, using \eg{} (for the height) results on branching random walks
\cite{Biggins76,Biggins77} 
and methods of  \cite{BroutinDevroye}, \cite{BroutinDevroyeEtAl2008}. 
However, for the \lpa{} trees, the height and
insertion depth are well known by other methods, see \eg{} \cite{Pittel},
\cite{SJ306}; hence we give instead another application. 

For a rooted tree $T$, let $h(v)$ denote the depth of a node $v$, \ie, its
distance to the root. Furthermore, for two nodes $v$ and $w$, let $v\land w$
denote their last common ancestor.
We define
\begin{equation}\label{Y}
Y=Y(T):=\sum_{v\neq w} h(v\land w),  
\end{equation}
summing over all pairs of distinct nodes. (For definiteness, we sum over
ordered pairs; summing over unordered pairs is the same except for a factor
$\frac12$. We may modify the definition by including the case $v=w$; this
adds the total pathlength which \as{} is of order $O(n\log n)$,
see \eqref{hh} below,
and thus  does not
affect our asymptotic result.)

The parameter $Y(T)$ occurs in various contexts. For example, if $\hW(T)$
denotes the Wiener index and $\hP(T)$ the total pathlength of $T$, then
$Y(T)=\hW(T)-(n-1)\hP(T)$, see \cite{SJ146}. Hence, for the random recursive
tree and binary search tree considered in \cite{Neininger-wiener},
the theorems there imply
convergence of $Y_n/n^2$ in distribution.
We extend this to convergence \as, and to all \lpa{} trees,
with  characterizations of the limit distribution $Q$ that are different
from the one given in \cite{Neininger-wiener}.

\begin{theorem}\label{TQ}
 Consider random split trees $T_n\VP$
of the type defined in the introduction
for some random split vector $\cP=(P_i)\ao$, and let $Y_n:=Y(T_n\VP)$ be given
by \eqref{Y}.
Assume that with positive probability,  $0<P_i<1$ for some $i$.
Then there exists a random variable $Q$ such that $Y_n/n^2\asto Q$ as \ntoo.
Furthermore, $Q$ has the representation in \eqref{Q} below
and satisfies
\begin{equation}\label{EQ0}
  \E Q= \frac{1}{1-\E\sum_i P_i^2}-1
<\infty,
\end{equation}
and the distributional fixed point equation
\begin{equation}\label{Qeq}
  Q\eqd\sumi P_i^2(1+Q^{(i)}),
\end{equation}
with all $Q^{(i)}$ independent of each other and of $(P_i)\ao$, and with
$Q^{(i)}\eqd Q$.

If $W$ is the size-biased splitting variable defined in \refS{Saux},
then also
\begin{equation}\label{EQW}
  \E Q
=\frac{\E W}{1-\E W}.
\end{equation}
\end{theorem}

Higher moments may be calculated from \eqref{Qeq} or \eqref{Q}, with some
effort. 

\begin{proof}
  We modify the definition of split trees by never placing a ball in an node; 
we use rule \ref{split2} for all nodes, and thus each ball travels along an
infinite path, chosen randomly with probabilities determined by the split
vectors at the visited nodes.
Let $X_{k,i}$ be the number of the child chosen by ball $k$ at the $i$th
node it visits, and let $\bX_k:=(X_{k,i})_{i=1}^\infty$.
Label the nodes of $\cT_\infty$ by strings in $\bbN^*$ as in
\refR{Rlabel}. Then the path of ball $k$ is $\emptyset$, $X_{k,1}$,
$X_{k,1}X_{k,2}$, \dots, visiting the nodes labelled by initial segments of
$\bX_k$. 
Note that conditioned on the split vectors $\cV\vv$ for all $v\in\cT_\infty$,
the sequences $\bX_k$ are \iid{} random infinite sequences with the
distribution
\begin{equation}\label{emil}
  \P(X_{k,j}=i_j,\, 1\le j\le m)  
=\prod_{j=1}^m P_{i_j}^{(i_1\dotsm i_{j-1})}.
\end{equation}

For two sequences $\bX,\bX'\in\bbN^\infty$, let
\begin{equation}
  f(\bX,\bX'):=\min\set{i:X_i\neq X'_i}-1,
\end{equation}
\ie, the length of the longest common initial segment.
Let $v_k$ be the node in $T_n$ that contains ball $k$, and note that if
neither $v_k$ nor $v_\ell$ is an ancestor of the other, then 
$h(v_k\land v_\ell)=f(\bX_k,\bX_\ell)$. 

We define, as an approximation of $Y_n$, 
\begin{equation}\label{hY}
  \hY_n:=\sum_{k,\ell\le n,\; k\neq\ell}f(\bX_k,\bX_\ell)
=2\sum_{\ell<k\le n}f(\bX_k,\bX_\ell).
\end{equation}

Condition on all split vectors $\cP\vv$. Then, using \eqref{emil},
\begin{equation}\label{Q}
  \begin{split}
&  \E\bigpar{f(\bX_1,\bX_2)\mid \set{\cP\vv, v\in\cT_\infty}}
\\&\qquad
=\E \sum_{m=1}^\infty \sum_{i_1,\dots,i_m\in\bbN} 
\ett{X_{1,j}=X_{2,j}=i_j  \text{ for } j=1,\dots,m}
\\&\qquad
= \sum_{m=1}^\infty \sum_{i_1,\dots,i_m\in\bbN} 
\Bigpar{\prod_{j=1}^m P_{i_j}^{(i_1\dotsm i_{j-1})}}^2
=:Q.
  \end{split}
\end{equation}
Hence, since the split vectors are \iid,
\begin{equation}\label{EQ}
  \begin{split}
\E f(\bX_1,\bX_2)
&= \E Q 
= \sum_{m=1}^\infty \sum_{i_1,\dots,i_m\in\bbN} 
\prod_{j=1}^m \E P_{i_j}^2
=\summ \Bigpar{\sum_i \E P_i^2}^m
\\&
= \frac{1}{1-\sum_i \E P_i^2}-1.
  \end{split}
\end{equation}
Since $\sum_i P_i^2\le\sum_i P_i=1$, with strict inequlity with positive
probability, 
$\E \sum_i P_i^2<1$, and thus \eqref{EQ} shows that $\E f(\bX_1,\bX_2)<\infty$.
Consequently, \as,
\begin{equation}\label{efa}
Q= \E\bigpar{f(\bX_1,\bX_2)\mid \set{\cP\vv, v\in\cT_\infty}}<\infty. 
\end{equation}

Condition again on all split vectors $\cP\vv$. Then
 the random sequences $\bX_k$ are \iid, and thus \eqref{hY} is a $U$-statistic.
Hence, we can apply 
the strong law of large numbers for $U$-statistics by \citet{Hoeffding},
which shows that \as{}
\begin{equation}
  \label{hylim}
\frac{\hY_n}{n(n-1)}
\to
  \E\bigpar{f(\bX_1,\bX_2)\mid \set{\cP\vv, v\in\cT_\infty}}=Q.
\end{equation}
Consequently, also unconditionally,
\begin{equation}
  \label{hylim2}
\frac{\hY_n}{n(n-1)}
\asto Q.
\end{equation}

It remains only to prove that $(\hY_n-Y_n)/n^2\asto0$, since we already have
shown \eqref{EQ0}, 
which implies \eqref{EQW} by \eqref{sizebiased},
and  \eqref{Qeq} follows from the representation \eqref{Q}.

As noted above, if $\ell<k$, then $h(v_k\land v_\ell)=f(\bX_k,\bX_\ell)$
except possibly when $v_\ell$ is an ancestor of $v_k$; furthermore, in the
latter case 
\begin{equation}
  \label{up}
0\le h(v_k\land v_\ell)\le f(\bX_k,\bX_\ell).
\end{equation}
Let $H_n:=\max\set{ h(v):v\in T_n}$ be the height of $T_n=T_n\VP$, and let
$\xH_n:=\max\set{f(\bX_k,\bX_\ell):\ell<k\le n}$. Since a node $v_k$ has
at most $H_n$ ancestors, it follows from \eqref{up} that, writing $v\prec w$
when $v$ is ancestor of $w$,
\begin{equation}
  \label{puh}
0\le \hY_n-Y_n
= 2\sumkn \sum_{v_l\prec v_k} \bigpar{f(\bX_k,\bX_\ell) - h(v_k\land
  v_\ell)}
\le 2n H_n\xH_n.
\end{equation}
Furthermore, there is some node $v_k$ with $h(v_k)=H_n$, and if $v_\ell$ is
its parent, then $f(\bX_k,\bX_\ell)\ge H_n-1$; hence, $H_n\le\xH_n+1$.

Let $m=m_n:=\ceil{c \log n}$, where $c>0$ is a constant chosen later.
Then, arguing similarly to \eqref{Q}--\eqref{EQ},
\begin{equation}\label{Qm}
  \begin{split}
&  \P\bigpar{f(\bX_1,\bX_2)\ge m\mid \set{\cP\vv, v\in\cT_\infty}}
\\&\qquad
=\E \sum_{i_1,\dots,i_m\in\bbN} 
\ett{X_{1,j}=X_{2,j}=i_j  \text{ for } j=1,\dots,m}
\\&\qquad
=  \sum_{i_1,\dots,i_m\in\bbN} 
\Bigpar{\prod_{j=1}^m P_{i_j}^{(i_1\dotsm i_{j-1})}}^2
  \end{split}
\end{equation}
and thus, letting $a:=\sum_i\E P_i^2<1$,
\begin{equation}\label{EQm}
  \begin{split}
\P\bigpar{ f(\bX_1,\bX_2)\ge m}
=  \sum_{i_1,\dots,i_m\in\bbN} \prod_{j=1}^m \E P_{i_j}^2
=a^m.
  \end{split}
\end{equation}
By symmetry, we thus have
\begin{equation*}
  \P(\xH_n\ge m) \le \sum_{\ell<k\le n}\P\bigpar{f(\bX_k,\bX_\ell)\ge m} \le n^2 a^m
\le n^2 a^{c\log n} \le  n\qww,
\end{equation*}
provided we choose $c\ge 4/|\log a|$.  Consequently, by the Borel--Cantelli
lemma, \as{} $\xH_n\le m-1\le c\log n$ for all large $n$. 
Hence, \as{} for all large $n$, 
\begin{equation}\label{hh}
  H_n\le\xH+1\le c\log n+1,
\end{equation}
and
\eqref{puh}
shows that \as{} $\hY_n-Y_n=O(n\log^2 n)$.
In particular, $(\hY_n-Y_n)/n^2\asto0$, which as said above 
together with \eqref{hylim2}
completes the proof. 
\end{proof}

\begin{corollary}
  Let $Y_n:=Y(T_n\XP)$ be given by \eqref{Y} for the \lpa{} tree $T_n\XP$,
  and assume $\chi+\rho>0$.
Then $Y_n/n^2\asto Q$ for some random variable $Q$ with 
\begin{equation}\label{EQXP}
  \E Q = \frac{\rho}{\chi+\rho}.
\end{equation}
\end{corollary}
\begin{proof}
  Immediate by Theorems \ref{T1} and \ref{TQ}, using \eqref{EQW} and
  \eqref{hex} to obtain \eqref{EQXP}.
\end{proof}

\section{The case $\chi<0$: \maryit{s}}\label{S<0}

In this section we
consider the case $\chi<0$ of \lpa{} tree{s} further;
as noted above, this case has some special features.
By \refR{R1}, we may assume
$\chi=-1$, and then by our assumptions, $\rho>0$ is necessarily an integer,
say $\rho=m\in\bbN$. As said in \refR{R1}, the case
$m=1$ is trivial, with $T_n\XM1$ a path, so we are mainly interested
in $m\in\set{2,3,\dots}$.

By \eqref{chirho}, $w_m=0$, and thus no node in $T_n\XM{m}$ will get more
that $m$ children. In other words, the trees will all have outdegrees
bounded by $m$.
It follows from \refL{L2}, or directly from \eqref{gem}--\eqref{gemZ}, that
if, as in \refT{T1},
$(P_i)\ao\sim\GEM\bigpar{-\frac{1}{m-1},\frac{m}{m-1}}$, 
then
$P_j=0$ for $j>m$.
Consequently, in this case, the split tree can be defined using a finite
split vector $(P_j)_1^b$ as in Devroye's original definition (with $b=m$).

Recall than an \mary{} tree is a rooted tree where each node has at most $m$
children, and the children are labelled by distinct numbers in
\set{1,\dots,m};
in other words, a node has $m$ potential children, labelled $1,\dots,m$,
although not all of these have to be present. (Potential children that are
not nodes are known as external nodes.) The \mary{} trees can also be
defined as the subtrees of the infinite \mary{} tree $\cT_m$ that contain
the root. Note that \mary{} trees are ordered, but that the labelling
includes more information than just the order of children (for vertices of
degree less than $m$).

It is natural to regard the trees $T_n\XM m$  as $m$-ary trees by labelling the 
children of a node by $1,\dots,m$ in (uniformly) random order.
It is then easy to see that the construction above, with $w_k=m-k$ by
\eqref{chirho}, is equivalent to adding each new node at random uniformly over
all positions where it may be placed
in the infinite tree $\cT_m$, \ie, by converting a uniformly chosen random
external node to a node;
see \cite[Section 1.3.3]{Drmota}.
Regarded in this way, the
trees $T_n\XM m$ are called \maryit{s} (or $m$-ary recursive trees)
See also \cite[Example 1]{BergeronFS92}.

\begin{example}\label{EBST}
  The case $\chi=-1$, $m=2$ gives,
using the construction above with \mary{} (binary) trees and external nodes, 
the random binary search tree. 
As mentioned in the introduction,
the binary search tree was one of the original examples of random
  split trees in \cite{Devroye}, with the split vector $(U,1-U)$ where
  $U\sim U(0,1)$.

Our \refT{T1} also exhibits the binary search tree as a random split tree, 
but with split vector $(P_1,1-P_1)\sim\GEM(-1,2)$ and thus, by \eqref{gemZ},
$P_1=Z_1\sim B(2,1)$. There is no contradiction, since we consider the trees
as unordered in \refT{T1}, and thus 
any (possibly random) permutation of  the split vector yields the same
trees; in this case, it is easily seen that reordering $(P_1,P_2)$ uniformly
at random yields $(U,1-U)$. ($P_1\sim B(2,1)$ has density $2x$, and
$P_2=1-P_1$ thus density $2(1-x)$, leading to a density 1 for a uniformly
random choice of one of them.) 

There are many other split vectors yielding
the same unordered trees. For example, 
\refT{T1} gives $\PD(-1,2)$ as one of them.
By definition, $\PD(-1,2)$ is obtained by ordering $\GEM(-1,2)$ in
decreasing order; by the discussion above, this is equivalent to ordering
$(U,1-U)$ in decreasing order, and it follows that the split vector
$(\hP_1,\hP_2)\sim\PD(-1,2)$ has $\hP_1\sim U(\frac12,1)$ and $\hP_2=1-\hP_1$.

For the binary search tree, Devroye's original symmetric choice $(U,1-U)$
for the split vector has the advantage that, by symmetry, 
the random split tree then
coincides with the binary search tree also as  binary trees.
\end{example}

\begin{remark}
For $m>2$, the \maryit{} considered here is not the
same as the $m$-ary search tree; the latter is also a random split tree
\cite{Devroye}, but not of the simple type studied here.
\end{remark}

\refE{EBST} shows that when $m=2$, we may see the \maryit{} as a random split
tree also when regarded as an \maryt, and not only as an unordered tree as in
\refT{T1}. 
We show next that this extends to $m>2$.
Recall that the Dirichlet distribution $\Dir(\ga_1,\dots,\ga_m)$ is a
distribution of probability vectors $(X_1,\dots,X_m)$, \ie{} random vectors
with $X_i\ge0$ and $\sum_1^m X_i=1$; the distribution has 
the density function $cx_1^{\ga_1-1}\dotsm x_m^{\ga_m-1}\dd x_1\dotsm
\dd x_{m-1}$ with the normalization factor
$c=\gG(\ga_1+\dots+\ga_m)/\prod_1^m\gG(\ga_i)$. 

\begin{theorem}\label{T2}
  Let $m\ge2$.
The sequence of \maryit{s} $(T_n)\ao=(T_n\XM m)\ao$, considered as
$m$-ary trees, has
the same distribution as the sequence of random split trees $(T_n\VP)\ao$
with the split vector 
$\cP=(P_i)_1^m\sim \Dir(\frac{1}{m-1},\dots,\frac{1}{m-1})$.
\end{theorem}

\begin{proof}
By \refT{T1}, the sequence of \maryit{s} $(T_n\XM m)_n$ has, as unordered trees,
the same distribution as the random split trees $(T_n^{\cP'})_n$, where
$\cP'=(P'_i)\ao\sim \GEM\bigpar{-\frac{1}{m-1},\frac{m}{m-1}}$.
As noted above, $P'_j=0$ for $j>m$, so we may as well  use the finite
split vector $(P'_i)_1^m$. Let $\cP=(P_i)_1^m$ be a uniformly random
permutation of $(P'_i)_1^m$. Then, as sequences of unordered trees,
$(T_n^{\cP})_n\eqd (T_n^{\cP'})_n \eqd (T_n\XM m)_n$. 
Moreover, regarded as \mary{} trees, both 
$(T_n^{\cP})_n$ and $(T_n\XM m)_n$ are, by symmetry, invariant under random
relabellings of the children of each node. Consequently,
$(T_n^{\cP})_n\eqd  (T_n\XM m)_n$ also as \mary{} trees, as claimed.

It remains to identify the split vector $\cP$. The definition as a random
permutation of $(P_i')_1^m$ does not seem very convenient; instead we use a
variation of the argument in \refApp{AL2} for \refL{L2}. 
We may assume that $T_n=T_n\XM m=T_n\VP$, as \mary{} trees, for all $n\ge1$.
Let 
$N_j(n)$ be the number of nodes and
$N^e_j(n)$ the number of external nodes in the principal subtree $T_n^j$
(now using the given labelling of the children of the root).
It is easy to see that $N^e_j(n)=(m-1)N_j(n)+1$.

Consider first $T_n$ as 
the random split tree $T_n\VP$; then 
the law of large numbers yields, by conditioning on the split vector $\cP$ at
the root,
\begin{equation}\label{nu}
N_j(n)/n\asto P_j,
\qquad j=1,\dots,m.  
\end{equation}
Next, consider $T_n$ as the \maryit{} $T_n\XM m$, and 
regard the external nodes in $T_n^j$ as balls with colour $j$. 
Then the
external nodes evolve as a \Polya{} urn with $m$ colours, starting with one
ball of each colour and at each round adding $m-1$ balls of the same colour
as the drawn one. Then, 
see \eg{} \cite{Athreya1969} or \cite[Section 4.7.1]{JohnsonKotz},
the vector of proportions $\bigpar{N^e_j(n)/((m-1)n+1)}_{j=1}^m$ of the different
colours converges \as{} to a random vector with a symmetric Dirichlet
distribution  
$\Dir(\frac{1}{m-1},\dots,\frac{1}{m-1})$.
Hence the vector $\bigpar{N_j(n)/n}_j$ converges to the same limit.
This combined with \eqref{nu} shows that 
$\cP\sim\Dir(\frac{1}{m-1},\dots,\frac{1}{m-1})$.
\end{proof}

\begin{remark}
  If we modify the proof above by considering one $N_j$ at a time, using a
  sequence of two-colour \Polya{} urns as in \refApp{AL2}, we obtain a
  representation \eqref{gem} of the Dirichlet distributed split vector with
$Z_j\sim B\bigpar{\frac{1}{m-1},\frac{m-j}{m-1}}$, $j=1,\dots,m$; \cf{} the
similar but different \eqref{gemZ}.
(This representation can also be seen directly.)
\end{remark}

\begin{remark}
  \citet{BroutinDevroyeEtAl2008} study a general model of random  trees 
that generalizes split trees (with bounded outdegrees)
by allowing more general mechanisms to
  split the nodes (or balls) than the ones considered in the present paper.
(The main difference is that the splits only asymptotically are given by a
single split vector $\cV$.)
Their examples include the \maryit, and also increasing trees
as defined by \citet{BergeronFS92} with much more general weights, 
assuming only a finite maximum outdegree $m$; they show that 
some properties of such trees  asymptotically
depend only on $m$, and in particular that the
distribution of subtree sizes $\bigpar{N_j(n)/n}_1^d$ converges to
the Dirichlet distribution
$\Dir(\frac{1}{m-1},\dots,\frac{1}{m-1})$ seen also in \refT{T2} above.
(Recall that \refT{T2}, while for a special case only, is an
exact representation for all $n$ and not only an asymptotic result.)
\end{remark}

There is no analogue of \refT{T2} for $\chi\ge0$, since then the
split vector is infinite, and symmetrization is not possible.

\section*{Acknowledgement}
I thank Cecilia Holmgren for helpful discussions.

\appendix

\section{Two alternative proofs}\label{AA}

We give here two alternative arguments, a direct proof of \refL{L2} and 
an alternative version of part of the proof  of \refT{T1} without
using Kingman's theory of exchangeable partitions.
We do this both for completeness and because we find the alternative and
more direct arguments interesting.
(For the proof of \refT{T1}, it should be noted that the two 
arguments, although stated using different concepts,
 are closely related, see the proof of Kingman's paintbox theorem by
\citet[\S11]{Aldous}.) 

\subsection{A direct proof of \refL{L2}}\label{AL2}
We often write $N_k$ for $N_k(n)$.

Consider first the evolution of the first principal subtree $T^1_n$. 
Let us colour all nodes in $T^1_n$ red and all other nodes white.
If at some stage there are $r=N_1\ge1$ red nodes and $w$ white nodes, and thus
$n=r+w$ nodes in total, then the total weight $R$
of the red nodes is, using
\refL{L1}, 
\begin{equation}\label{R}
R=w(T^1_n)=r-\chi=N_1-\chi, 
\end{equation}
while the total weight of all nodes is
$w(T_n)=n-\chi$, and thus the total weight $W$ of the white nodes is
\begin{equation}\label{W}
  W=w(T_n)-w(T_n^1)=(n-\chi)-(r-\chi)=n-r=w. 
\end{equation}
By \eqref{R}--\eqref{W}, 
adding a new red node increases $R$ by 1, but does not
change $W$, while adding a new white node increases $W$ by 1 but does not
change $R$. Moreover, by definition, the 
probabilities that the next new node is red or white are proportional to $R$
and $W$. In other words, the total red and white
weights $R$ and $W$
evolve as a \Polya{} urn with balls of two colours, where a ball is
draw at random and replaced together with a new ball of the same colour.
(See \eg{} \cite{EggPol,Polya} and, even earlier, \cite{Markov}.)
Note that while the classical description of \Polya{} urns  considers
the numbers of balls of different colours, and thus implicitly assumes that
these are integers, the weights considered here may be arbitrary positive real
numbers; however, it has been noted many times that this extension of the
original definition does not change the results,
see \eg{} \cite[Remark 4.2]{SJ154} 
and \cf{} \cite{Jirina} for the related case of branching processes.

In our case, the first node is the root, which is white, and the second node
is its first child, which is the root of the principal subtree $T^1$ and
thus is red. Hence, the \Polya{} urn just described starts (at $n=2)$ with
$r=w=1$, and thus by \eqref{R}--\eqref{W} $R=1-\chi$ and  $W=1$.

It is well-known that for a \Polya{} urn of the type just described (adding one
new ball each time, of the same colour as the drawn one), with initial
(non-random) values $R_0$ and $W_0$ of the weights, the red proportion in
the urn, \ie, $R/(R+W)$, converges \as{} to a random variable $Z\sim
B(R_0,W_0)$. 
(Convergence in distribution follows easily from the simple exact formula
for the distribution of the sequence of the first $N $ draws \cite{Markov}; 
convergence
\as{} follows by the martingale convergence theorem, or by
exchangeability and de Finetti's theorem. 
See also \cite[Sections 4.2 and 6.3.3]{JohnsonKotz}.)
Consequently, in our case, $R/(R+W)\asto Z_1\sim B(1-\chi,1)$, and thus
by \eqref{R}--\eqref{W} $N_1(n)/n\asto Z_1\sim B(1-\chi,1)$.
Note that this is consistent with \eqref{gemZ}, with
$(\ga,\theta)=(\chi,\rho)$, since we assume \eqref{=1}.
Furthermore, by the definition \eqref{gem}, we have
$P_1=Z_1$, and thus $N_1(n)/n\asto P_1$.

We next consider $N_2$, then $N_3$, and so on. 
In general, for the $k$th principal subtree, we suppose by induction that
$N_i(n)/n\asto P_i$ for $1\le i<k$, with $P_i$ given by \eqref{gem} for some
independent random variables $Z_i$ satisfying \eqref{gemZ}, $i<k$.
We now colour all nodes in the principal subtree $T^k_n$ red, all nodes in
$T_n^1,\dots,T_n^{k-1}$ black, and the remaining ones white.
We then ignore all black nodes, and consider only the (random) times that a
new node is added and becomes red or white. Arguing as above, we see that
if there are $r=N_k\ge1$ red and $w$ white nodes, then the red and white
total weights $R$ and $W$ are given by
\begin{align}\label{Rk}
R&=w(T^k_n)=r-\chi=N_k-\chi, 
\\
\label{Wk}
  W&=w(T_n)-\sum_{i=1}^kw(T_n^i)=(n-\chi)-\sum_{i=1}^k(N_i-\chi)=w+(k-1)\chi.
\end{align}
Moreover, $(R,W)$ evolve as a \Polya{} urn as soon as there is a red node.
When the first red node appears, there is only one white node (the root),
since then $T^j$ is empty for $j>k$. Consequently, then $r=w=1$, and
\eqref{Rk}--\eqref{Wk} show that the \Polya{} urn now starts with
$R=1-\chi$ and $W=1+(k-1)\chi=k\chi+\rho$. 
Since the total number of non-black nodes is $n-\sum_{i<k}N_i$, it follows
that, as \ntoo,
\begin{equation}\label{poli}
  \frac{N_k(n)}{n-\sum_{i<k} N_i(n)}
\asto Z_k,
\end{equation}
for some random variable $Z_k\sim B(1-\chi,k\chi+\rho)$, again consistent
with \eqref{gemZ}.
Moreover, this \Polya{}
urn is independent of what happens inside the black subtrees, and thus $Z_k$
is independent of $Z_1,\dots,Z_{k-1}$.
We have, by \eqref{poli}, the inductive hypothesis and \eqref{gem},
\begin{equation}
  \begin{split}
  \frac{N_k(n)}{n}
&=
  \frac{N_k(n)}{n-\sum_{i<k} N_i(n)}\cdot \frac{n-\sum_{i<k} N_i(n)}{n}
\\&
\asto Z_k\Bigpar{1-\sum_{i<k} P_i}
= Z_k\prod_{i<k}(1-Z_i)=P_k .    
  \end{split}
\end{equation}
This completes the proof.
\qed 

\subsection{An alternative argument in the proof of \refT{T1}}\label{AT}
The equality \eqref{exch} shows a kind of limited exchangeability 
for the infinite sequence $(X_i)_1^\infty$; 
limited because we only consider acceptable sequences, \ie, the
first appearance of each label is in the natural order.
We eliminate this restriction by a random relabelling of the principal subtrees;
let $(U_i)_1^\infty$ be an \iid{} sequence of $U(0,1)$ random variables,
independent of everything else, and relabel the balls passed to subtree $i$
by $U_i$. Then the sequence of new labels is $(U_{X_i})_1^\infty$, and it
follows from \eqref{exch} and symmetry that this sequence is exchangeable,
\ie, its distribution is invariant under arbitrary permutations.
Hence, by de Finetti's theorem \cite[Theorem 11.10]{Kallenberg},
there exists a random probability measure $\bP$ on $\oi$ such that
the conditional distribution of $(U_{X_i})_1^\infty$ given $\bP$ \as{}
equals the distribution of an \iid{} sequence of random variables with the
distribution $\bP$.

As in the proof in \refS{Spf},
every principal subtree $T^j$ satisfies 
by \refL{L2}
either $|T^j(n)|\to\infty$
as \ntoo, or $T^j(n)=\emptyset$  for all $n$.
Hence, \as{} there exists some (random) index $\ell$ such that $X_\ell=X_1$,
and thus $U_{X_\ell}=U_{X_1}$. It follows that the random measure $\bP$
\as{} has no continuous part, so $\bP=\sum_{i=1}^\infty \aP_i\gd_{\xi_i}$, for some
random variables $\aP_i\ge0$ and (distinct) 
random points $\xi_i\in\oi$, with $\sum_i \aP_i=1$. (We allow $\aP_i=0$, and can
thus write $\bP$ as an infinite sum even if its support happens to be
finite.) 

The labels $\xi_i$ serve only to distinguish the subtrees, and we may now
relabel again, replacing $\xi_i$ by $i$. After this relabelling, the
sequence $(X_i)$ has become a sequence which conditioned on
$\caP:=(\aP_i)_1^\infty$ is an \iid{} sequence with each variable having the
distribution $\caP$. In other words, up to a (random) permutation of the
children, the rules
\ref{pax1}--\ref{pax2} yield the same result as the split tree
rules \ref{split1}--\ref{split2} given in the introduction, using the split
vector $\caP=(\aP_i)_1^\infty$.

It remains to identify this split vector, which is done as in \refS{Spf}, using
\eqref{LLN} and \refL{L2}.
\qed 

\newcommand\AAP{\emph{Adv. Appl. Probab.} }
\newcommand\JAP{\emph{J. Appl. Probab.} }
\newcommand\JAMS{\emph{J. \AMS} }
\newcommand\MAMS{\emph{Memoirs \AMS} }
\newcommand\PAMS{\emph{Proc. \AMS} }
\newcommand\TAMS{\emph{Trans. \AMS} }
\newcommand\AnnMS{\emph{Ann. Math. Statist.} }
\newcommand\AnnPr{\emph{Ann. Probab.} }
\newcommand\CPC{\emph{Combin. Probab. Comput.} }
\newcommand\JMAA{\emph{J. Math. Anal. Appl.} }
\newcommand\RSA{\emph{Random Structures Algorithms} }
\newcommand\ZW{\emph{Z. Wahrsch. Verw. Gebiete} }
\newcommand\DMTCS{\jour{Discr. Math. Theor. Comput. Sci.} }

\newcommand\AMS{Amer. Math. Soc.}
\newcommand\Springer{Springer-Verlag}
\newcommand\Wiley{Wiley}

\newcommand\vol{\textbf}
\newcommand\jour{\emph}
\newcommand\book{\emph}
\newcommand\inbook{\emph}
\def\no#1#2,{\unskip#2, no. #1,} 
\newcommand\toappear{\unskip, to appear}

\newcommand\arxiv[1]{\texttt{arXiv:#1}}
\newcommand\arXiv{\arxiv}

\def\nobibitem#1\par{}


\begin{thebibliography}{99}

\bibitem[Aldous(1985)]{Aldous}
David J. Aldous.
Exchangeability and related topics. 
\emph{{\'E}cole d'{\'E}t{\'e} de Probabilit{\'e}s de Saint-Flour XIII -- 
1983}, 1--198,
Lecture Notes in Math. 1117, Springer, Berlin, 1985. 

\bibitem{Athreya1969}
Krishna B. Athreya.
On a characteristic property of Polya's urn.
\emph{Studia Sci. Math. Hungar.} \textbf4 (1969), 31--35. 

\bibitem[Barab\'asi and Albert(1999)]{BarabasiA}
Albert-L\'aszl\'o Barab\'asi and R\'eka Albert.
Emergence of scaling in random networks.
\emph{Science} 
\vol{286} (1999), no. 5439,  509--512.

\bibitem[Bergeron,  Flajolet and Salvy(1992)]{BergeronFS92}
Fran\c cois Bergeron, Philippe Flajolet and Bruno Salvy.
Varieties of increasing trees. 
\emph{CAAP '92 (Rennes, 1992)}, 24--48,
Lecture Notes in Comput. Sci. 581, Springer, Berlin, 1992.

\bibitem[Bertoin(2006)]{Bertoin}
Jean Bertoin.
\book{Random Fragmentation and Coagulation Processes}.
Cambridge Univ. Press, Cambridge, 2006.

\bibitem[Biggins(1976)]{Biggins76}
J. D. Biggins. 
The first- and last-birth problems for a multitype age-dependent branching
process. 
\emph{Advances in Appl. Probability} \textbf{8} (1976), no. 3, 446--459.


\bibitem[Biggins(1977)]{Biggins77}
J. D. Biggins. 
Chernoff's theorem in the branching random walk. 
\emph{J. Appl. Probability} \vol{14} (1977), no. 3, 630--636.


\bibitem{BroutinDevroye}
Nicolas Broutin and Luc Devroye.
Large deviations for the weighted height of an extended class of trees.
\emph{Algorithmica} \textbf{46} (2006), no. 3-4, 271--297. 

\bibitem[Broutin et al.(2008)]{BroutinDevroyeEtAl2008}
Nicolas Broutin, Luc Devroye, Erin McLeish and Mikael de la Salle. 
The height of increasing trees.
\emph{Random Structures Algorithms} \textbf{32} (2008), no. 4, 494--518. 

\bibitem{BroutinH}
Nicolas Broutin and Cecilia Holmgren.
The total path length of split trees.
\emph{Ann. Appl. Probab.} \textbf{22} (2012), no. 5, 1745--1777. 

\bibitem[Devroye(1999)]{Devroye}
Luc Devroye.
Universal limit laws for depths in random trees.
\emph{SIAM J. Comput.} \textbf{28} (1999), no. 2, 409--432. 

\bibitem[Drmota(2009)]{Drmota}
Michael Drmota.
\emph{Random Trees}.
Springer, Vienna, 2009.

\bibitem[Eggenberger and \Polya(1923)]{EggPol}
F. Eggenberger and George \Polya.  
\"Uber die Statistik verketteter Vorg\"ange.
\jour{Zeitschrift Angew. Math. Mech.}
\vol3 (1923), 279--289.

\bibitem[Hoeffding(1961)]{Hoeffding}
Wassily Hoeffding.
The strong law of large numbers for   $U$-statistics.
Institute of Statistics, Univ. of North Carolina, Mimeograph
      series 302 (1961).
\url{https://repository.lib.ncsu.edu/handle/1840.4/2128}

\bibitem{Holmgren}
Cecilia Holmgren.
Novel characteristic of split trees by use of renewal theory. 
\emph{Electron. J. Probab.} \textbf{17} (2012), no. 5, 27 pp. 

\bibitem{SJ306}
Cecilia Holmgren and Svante Janson.
Fringe trees, Crump--Mode--Jagers branching processes and $m$-ary search
trees.
\emph{Probability Surveys} \textbf{14} (2017), 53--154.

\bibitem[Janson(2003)]{SJ146}
Svante Janson,
The Wiener index of simply generated random trees.
\RSA \vol{22}\no4 (2003), 337--358. 

\bibitem[Janson(2004)]{SJ154}  
Svante Janson.
Functional limit theorems for multitype branching processes and generalized
\Polya{} urns.  
\emph{Stoch. Process. Appl.} \textbf{110} (2004),  177--245.

\bibitem{Jirina}
Miloslav Ji\v rina.
Stochastic branching processes with continuous state space.  
\jour{Czechoslovak Math. J.} \vol{8 (83)}  (1958), 292--313.

\bibitem{JohnsonKotz}
Norman L. Johnson and Samuel  Kotz.
Urn models and their application. 
John Wiley \& Sons, New York, 
1977. 

\bibitem{Kallenberg}
Olav Kallenberg.
\book{Foundations of Modern Probability.}
2nd ed., Springer, New York, 2002. 

\bibitem[Kingman(1978)]{Kingman-partition}
John F. C. Kingman. 
The representation of partition structures.
\emph{J. London Math. Soc. (2)} \vol{18} (1978), no. 2, 374--380. 

\bibitem[Kingman(1982)]{Kingman-coalescent}
John F. C. Kingman. 
The coalescent.
\emph{Stochastic Process. Appl.} \textbf{13} (1982), no. 3, 235--248. 

\bibitem[Markov(1917)]{Markov} 
A. A. Markov.
Sur quelques formules limites du calcul des probabilit\'es
(Russian).
\emph{Bulletin de l'Acad\'emie Imp\'eriale des Sciences, Petrograd}
\textbf{11} (1917), no. 3, 177--186.

\bibitem{Neininger-wiener}
 Ralph Neininger. 
The Wiener index of random trees. 
\emph{Combin. Probab. Comput.} \textbf{11} (2002), no. 6, 587--597.

\bibitem{NIST}
\emph{NIST Handbook of Mathematical Functions}. 
Edited by Frank W. J. Olver, Daniel W. Lozier, Ronald F. Boisvert and
Charles W. Clark. 
Cambridge Univ. Press, 2010. \\
Also available as 
\emph{NIST Digital Library of Mathematical Functions},
\url{http://dlmf.nist.gov/}

\bibitem[Panholzer and  Prodinger(2007)]{PP-verysimple}
Alois Panholzer and Helmut Prodinger. 
Level of nodes in increasing trees revisited. 
\emph{Random Structures Algorithms} \textbf{31} (2007), no. 2, 203--226.  


\bibitem[Pitman(2006)]{Pitman}
Jim Pitman. 
\emph{Combinatorial Stochastic Processes}. 
{\'E}cole d'{\'E}t{\'e} de Probabilit{\'e}s de Saint-Flour
XXXII -- 2002.
Lecture Notes in Math. 1875, Springer, Berlin, 2006. 

\bibitem[Pittel(1994)]{Pittel}
Boris Pittel. 
Note on the heights of random recursive trees and random $m$-ary search
trees. 
\emph{Random Structures Algorithms} \textbf5 (1994), no. 2, 337--347.

\bibitem[P\'olya(1931)] {Polya}
George P\'olya. 
Sur quelques points de la th\'eorie des probabilit\'es.
\jour{Ann. Inst. Poincar\'e} \vol1 (1930),
117--161.

\bibitem[Szyma{\'n}ski(1987)]{Szymanski}
Jerzy Szyma{\'n}ski.
On a nonuniform random recursive tree. 
\emph{Annals of Discrete Math.} \vol{33} (1987), 
297--306.


\end{thebibliography}
\end{document}